\documentclass[12pt]{amsart}

\usepackage{fullpage}
\usepackage{mathtools}
\usepackage{amssymb,amsmath,amsthm,amscd,mathrsfs,graphicx}
\usepackage[dvipsnames]{xcolor}
\usepackage{bm}
\usepackage{dsfont}
\usepackage{enumerate}
\usepackage{epsfig}
\usepackage{float, graphicx}
\usepackage{latexsym, amsxtra}
\usepackage{pdfpages}
\usepackage{multicol}
\usepackage[normalem]{ulem}
\usepackage{psfrag}
\usepackage{stmaryrd}
\usepackage{tikz}
\usepackage[T1]{fontenc}
\usepackage{url}
\usepackage{verbatim}
\usepackage{xy}
\usepackage{indentfirst}
\usepackage{tikz-cd}
\tikzset{%
    symbol/.style={%
        ,draw=none
        ,every to/.append style={%
            edge node={node [sloped, allow upside down, auto=false]{$#1$}}}
    }
}
\usepackage{url}
\usepackage{longtable}

\usepackage[
  colorlinks=true,
  linkcolor=red,
  urlcolor=blue,
  citecolor=blue
]{hyperref}

\makeatletter
\def\thm@space@setup{%
  \thm@preskip=2ex \thm@postskip=2ex
}
\makeatother

\makeatletter

\newcommand{\Rmnum}[1]{\expandafter\@slowromancap\romannumeral #1@}
\makeatother

\numberwithin{equation}{section}
\theoremstyle{plain}

\newtheorem{thm}{Theorem~}[section] 
\newtheorem{lem}[thm]{Lemma~}

\newtheorem{prop}[thm]{Proposition~}

\newtheorem{cor}[thm]{Corollary~}

\theoremstyle{remark}
\newtheorem{rmk}[thm]{Remark~}
\newtheorem{ex}[thm]{Example~}

\theoremstyle{definition}

\newenvironment{thmbis}[1]
  {\renewcommand{\thethm}{\ref{#1}$'$}%
   \addtocounter{thm}{-1}%
   \begin{thm}}
  {\end{thm}}

\newcommand{\CC}{\mathbb{C}}
\newcommand{\ZZ}{\mathbb{Z}}

\newcommand{\PP}{\mathbb{P}}

\newcommand\PGL{\mathrm{PGL}}

\newcommand\rank{\mathrm{rank}}

\newcommand\SL{\mathrm{SL}}

\newcommand\diag{\mathrm{diag}}
\newcommand\GL{\mathrm{GL}}

\title{Several Sufficient conditions for projective hypersurfaces to be GIT (semi-)stable}

 \author[X. He]{Xuancong He}
\address{Fudan University, China}
\email{xche25@m.fudan.edu.cn}

\date{}

\begin{document}
\bibliographystyle{amsalpha}

\begin{abstract}
In this paper, I present some sufficient conditions for projective hypersurfaces to be GIT (semi-)stable. These conditions will be presented in terms of dimension and degree of the hypersurfaces, dimension of the singular locus and multiplicities of the singular points. When singularities of the hypersurface are isolated and all have multiplicity 2, we can judge its stability via the ranks of Hessian matrices at these singular points. 
\end{abstract}

\maketitle

\emph{Notation}:
\begin{enumerate}
\item $V:= \CC[x_0,..,x_n]_d$ is the vector space of degree $d$ homogeneous polynomials in $n+1$ variables. We always assume $n \ge 2$ and $d \ge 3$. 
\item $f \in V$ is a homogeneous polynomial and $H = V(f)$ is the hypersurface determined by its zero set. 
\item $\vec{r} \in \ZZ^{n+1}$ is such that $r_0\ge..\ge r_n$, $\vec{r} \ne \vec{0}$ and $\Sigma_jr_j = 0$. 
\item $M_{\ge0(>0)}(\vec{r}) = \{f \in \Sigma a_{i_0,..,i_n}x_0^{i_0}..x_n^{i_n}\in \CC[x_0,..,x_n]_d|\Sigma_j r_ji_j\ge(>)0, \forall a_{i_0,..,i_n}\ne 0\}$. 
\end{enumerate}

\section{Introduction}
Geometric invariant theory (GIT) is an important tool in algebraic geometry to construct a quotient of an algebraic variety $X$ by an algebraic action of a linear algebraic group $G$. Since many moduli spaces are constructed as quotients of parameter spaces by actions that identify equivalent objects, GIT is also an important tool to construct and compactify moduli spaces. One of the most important examples is the construction of the moduli space of projective hypersurfaces. 

Let $V := \CC[x_0,..,x_n]_d$ be the vector space of degree $d$ homogeneous polynomials in $n+1$ variables. The parameter space or the Hilbert scheme of degree $d$ hypersurfaces of dimension $n-1$ is the projective space $\PP(V)$. To construct the moduli space of degree $d$ hypersurfaces of dimension $n-1$ up to projective equivalence, we need to construct a quotient of $\PP(V)$ by $\PGL(n+1)$. Ideally, we would want a map from $\PP(V)$ to a quotient (of set) $\PP(V)/\PGL(n+1)$, which is expected to be an algebraic variety. However, this would imply that inverse image of each point, which is an orbit, is closed. However, this does not happen in the case where $\PGL(n+1)$ acts on $\PP(V)$. A possible approach to solving this problem is to choose an invariant Zariski open subset as large as possible, such that the geometric quotient exists. In fact, all the points in $\PP(V)$ that have closed orbits and dimension $0$ stabilizers in $\PGL(n+1)$ will constitute an open subset that suits the need, they are called stable points, denoted $\PP(V)^{s}$. However, another problem emerges, $\PP(V)^{s}/\PGL(n+1)$ may not be projective even though $\PP(V)$ is projective. The solution is to not restrict ourselves to geometric quotients, but turn to categorical quotients which might identity different orbits to the quotient. In this point of view, we only throw away those points that cannot lie in invariant Zariski open subsets, which are called unstable points. The complement of unstable points is an open subset containing $\PP(V)^{s}$, called the semi-stable points, and denoted as $\PP(V)^{ss}$. 

It is one of the main problems in GIT to determine semi-stable and stable points. According to \cite[Proposition 4.2]{mumford1994geometric}, when $d \ge 3$, non-singular hypersurfaces are stable points. Therefore, it is a natural question to ask, for a hypersurface to remain (semi-)stable, what singularities are allowed to occur on it. However, there is no complete answer to this problem for general $(n,d)$ but for several small values (for notation, let $f \in V$ be a homogeneous polynomial and $H = V(f)$ be the hypersurface determined by the zero set of $f$. )
\begin{itemize}
    \item When $d=1$, there are no semi-stable points, see \cite[Example 7.7]{hoskins2015moduli}. When $d = 2$, $H$ is semi-stable if and only if it is smooth, and it is never stable, see \cite[Example 10.1]{dolgachev2003lectures}. 
    \item When $d = 3$, various cases where $n$ is small have been studied clearly. When $n = 2$, a cubic curve is semi-stable if and only if it has at most $A_1$ singularities, it is stable if and only if it is smooth, see \cite[Lemma 7.25]{hoskins2015moduli}; When $n=3$, a cubic surface is semi-stable if and only if it has at most $A_1$ and $A_2$ singularities, it is stable if and only if it has at most $A_1$ singularities, see \cite[Theorem 7.14, Theorem 7.20]{mukai2003introduction}. When $n = 4, 5$, the computation becomes much more complicated, see \cite{yokoyama2002stability} and \cite{laza2009moduli}. 
    \item When $(n,d) = (2,4)$, see the table on page 80 in \cite{mumford1994geometric}. When $(n,d) = (3,4)$, see \cite{shah1981degenerations}. 
\end{itemize}

In a recent paper of Thomas Mordant \cite{mordant2024note}, he provides some sufficient conditions to determine stability of general hypersurfaces, which is also the aim of my paper. 

\begin{thm}\cite[Theorem 1.1]{mordant2024note}
\label{theorem:mordantmain}
    The base field $k$ is an algebraically closed field. Let $\delta$ be the maximal multiplicity of $H$ at a point of $H(k)$, and let $s$ be the dimension of the singular locus $H_{sing}$ of $H$. 
    \begin{enumerate}
        \item If the following condition holds: 
        \[
        d \ge \delta\min(n+1,s+3) \ \ (resp.\  d>\delta\min(n+1,s+3)),
        \]
        then $H$ is semi-stable (resp. stable). 
        \item Asssume $N \ge 2$. If for every point $P \in H(k)$ where $H$ has multiplicity $\delta$, the projective tangent cone $\PP(C_PH)$ in $\PP(T_P\PP_k^n)\cong\PP_k^{n-1}$ is not the cone over some hypersurface in a projective hyperplane of $\PP_k^{n-1}$, and if the following condition holds:
        \[
        d \ge (\delta-1)\min(N+1, s+3)\ \ (resp.\ d>(\delta -1)\min(N+1,s+3)), 
        \]
        then $H$ is semi-stable (resp. stable). 
    \end{enumerate}
\end{thm}
\begin{rmk}
    Note that we must have $s \le n-1$ and thus $s+3 \le n+2$. Therefore, the lower bound is $\delta(s+3)$ or $(\delta-1)(s+3)$ most of the time unless $s$ reaches its maximal $n-1$. 
\end{rmk}

When the hypersurfaces have at most isolated singularities, i.e. $s = 0$, there are two corollaries. 
\begin{cor}\cite[1.2.2]{mordant2024note}
\label{corollary:mordant1}
    When $H$ has at most isolated singularities, if further
    \[
    n \ge 2,\ \ d\ge3\delta\ \ (resp.\ d>3\delta),
    \]
    then $H$ is semi-stable (resp. stable). 
\end{cor}
\begin{rmk}
    When $H$ is singular, we must have $\delta \ge 2$, and therefore this corollary mainly applies to the case where $d \ge 6$. 
\end{rmk}
\begin{cor}\cite[1.2.3]{mordant2024note}
\label{corollar:mordant2}
    When $H$ has at most isolated singularities, and the projective tangent cones at these singularities are not the cones over some hypersurfaces in a projective hyperplane of $\PP_k^{n-1}$, if further
    \[
    n \ge 2,\ \ d\ge3(\delta - 1)\ \ (resp.\ d>3(\delta-1)),
    \]
    then $H$ is semi-stable (resp. stable). 
\end{cor}
\begin{rmk}
    When $n \ge 2$ and $\delta = 2$, the condition in the corollary is the same as requiring $H$ to have at most $A_1$ singularities. Therefore, when $H$ has at most $A_1$ singularities, and further
    \[
    d \ge 3\ \ (resp.\ d\ge 4), 
    \]
    then $H$ is semi-stable (stable). 
\end{rmk}

Although I will present results similar to \cite{mordant2024note}, it does not mean that my work here is worthless. In fact, I will give a better lower bound compared to theorem \ref{theorem:mordantmain} (1), and use examples to illustrate that when $s=0$ and $\delta=2$, this bound is sharp. Meanwhile, when $s=0$ and $\delta = 2$, a new way to determine stability of hypersurfaces is presented here, via the ranks of Hessian matrices at these singular points. 

It seems good to cite Thomas Mordant's work and then state new things after his results. However, my way of tackling this problem is a bit different and thus I choose to build up the proof from the beginning for my own convenience. 

\section{Main Results}
For convenience, we assume the base field to be $\CC$. Most arguments will apply to other base fields without change. $n$ is still the dimension of the projective space, $d$ is the degree of the hypersurface $H$, $s$ is the dimension of the singular locus $H_{sing}$ and $\delta$ is the maximal multiplicities of singularities of $H$. In the following, we always assume $n \ge 2$ and $d \ge 3$. 
\begin{thm}
\label{theorem:main1}
    When the hypersurface $H$ satisfies one of the following conditions
    \begin{itemize}
        \item $s \le n-2$ and
        \[
        \delta < \frac{d(d-2)}{(s+2)d-(s+3)}\ \ (resp.\   \delta \le \frac{d(d-2)}{(s+2)d-(s+3)})
        \]
        \item $s = n-1$ and
        \[
        \delta < \frac{d}{n+1}\ \ (resp.\ \delta\le\frac{d}{n+1})
        \]
        then $H$ is semi-stable (resp. stable). 
    \end{itemize}
\end{thm}

In fact, since $d \ge 3$, we can state this theorem in an equivalent way, similar to \ref{theorem:mordantmain}
\begin{thmbis}{theorem:main1}
    When the hypersurface $H$ satisfies one of the following conditions
    \begin{itemize}
        \item $s \le n-2$ and
            \begin{align*}
            & d > \frac{\delta(s+2)}{2} + \sqrt{(\frac{\delta(s+2)}{2})^2-\delta+1}+1\\
            & (resp.\ d \ge \frac{\delta(s+2)}{2} + \sqrt{(\frac{\delta(s+2)}{2})^2-\delta+1}+1)
            \end{align*}
        \item $s = n-1$ and
        \[
        d > \delta(n+1)\ \ (resp.\ d \ge \delta(n+1))
        \]
        then $H$ is stable (resp. semi-stable). 
    \end{itemize}
\end{thmbis}
\begin{rmk}
    Note that here we give a better lower bound compared to \ref{theorem:mordantmain} (1), since
    \[
    \frac{\delta(s+2)}{2} + \sqrt{(\frac{\delta(s+2)}{2})^2-\delta+1}+1 < \frac{\delta(s+2)}{2} + \sqrt{(\frac{\delta(s+2)}{2})^2}+1 = \delta(s+2)+1<\delta(s+3).
    \]
\end{rmk}

When $H$ has at most isolated singularities, and let $\delta = 2$, we get
\begin{cor}
\label{corollary:main1}
    Suppose $d \ge 5$, if $H$ has at most isolated singularities of multiplicity 2, then it is stable. 
\end{cor}
\begin{rmk}
    In fact, $d = 5$ is the best bound in the case where $s = 0$ and $\delta = 2$. More precisely, when $d = 3$ or $4$, for each $n \ge 2$, we can find a hypersurface $H$ with at most isolated singularities of multiplicity 2 that is not semi-stable. 
\end{rmk}
\begin{ex}
    Consider
    \[
    f_n = x_0^2x_n + x_1^3 +..+x_{n-1}^3, 
    \]
    its only singularity is $[0:..:0:1]$, with multiplicity $2$. If we take $\vec{r} = (3(n-1),1,..,1,-4(n-1))$, then $f_n \in M_{>0}(\vec{r})$. According to \S \ref{HMC}, $H = V(f_n)$ is not semi-stable. 
\end{ex}
\begin{ex}
    Consider
    \[
    g_n = x_0^2x_n^2 + x_0x_{n-1}^3 + x_1^4 + .. + x_{n-2}^4,
    \]
    its only singularities are the two points $[1:0:..:0]$ and $[0:..:0:1]$, each with multiplicity $2$. If we let $\vec{r} = (3n+2,1,..,1,-n,-3n)$, then $g_n \in M_{>0}(\vec{r})$ and therefore is not semi-stable. 
\end{ex}

We need to say more about stability of degree $3$ and $4$ hypersurfaces when they have at most isolated singularities of multiplicity $2$. 

\begin{thm}
\label{theorem:main2}
    When $d = 3$ or $d = 4$, suppose $H$ has at most isolated singularities of multiplicity $2$, let $r$ be the minimal of the ranks of Hessian matrices at these singularities. Suppose
    \[
    r > \frac{2(n+1)}{d}\ \ (resp.\ r \ge \frac{2(n+1)}{d}), 
    \]
    then $H$ is stable (resp. semi-stable). 
\end{thm}
\begin{rmk}
    If we set $cr = n-r$, i.e. $cr$ is the maximal of the coranks of Hessian matrices at these singularities, then we have
    \begin{itemize}
        \item When $d = 3$, suppose $cr < \frac{n-2}{3}$ (resp. $cr \le \frac{n-2}{3}$), then $H$ is stable (resp. semi-stable). 
        \item When $d = 4$, suppose $cr < \frac{n-1}{2}$ (resp. $cr \le \frac{n-1}{2}$), then $H$ is stable (resp. semi-stable). 
    \end{itemize}
\end{rmk}

Now various interesting results can be deduced from theorem \ref{theorem:main2}
\begin{cor}
\label{corollary:A1}
    When $n \ge 2$ and $d \ge 3$, hypersurfaces with at most $A_1$ singularities are semi-stable. If furthermore, $n \ge 3$ or $d \ge 4$, then hypersurfaces with at most $A_1$ singularities are stable. 
\end{cor}
\begin{proof}
    The $d \ge 5$ case follows from corollary \ref{corollary:main1}. The $d = 3$ or $d = 4$ case follows form theorem \ref{theorem:main2} since $A_1$ singularities are exactly those singularities with Hessian matrices having corank $0$. 
\end{proof}
\begin{rmk}
    Note that this is a stronger version of the remark after corollary \ref{corollar:mordant2}. In fact, corollary \ref{corollary:A1} completely answers the question that assuming what values of $(n,d)$ will we have the conclusion that hypersurfaces with at most $A_1$ singularities are stable or semi-stable. 
\end{rmk}
\begin{cor}
    When $d = 3$ and $n \ge 5$, hypersurfaces with at most $A_m\ (m \ge 1)$ singularities are stable. When $d = 4$ and $n \ge 3$, hypersurfaces with at most $A_m\ (m\ge 1)$ singularities are stable. 
\end{cor}
\begin{proof}
    Note that $A_m\ (m\ge 1)$ singularities are just isolated singularities of corank $1$. The case where $(n,d) = (5,3)$ can be deduced from \cite[Theorem 1.1]{laza2009moduli}. The case where $(n,d) = (3,4)$ can be deduced from \cite[Theorem 2.4]{shah1981degenerations}. The rest are direct consequences of theorem \ref{theorem:main2}. 
\end{proof}
\begin{cor}
\label{corollary:cr2}
    When $d = 3$ and $n >(\ge) 8$, hypersurfaces with at most isolated singularities of corank $\le 2$ are (semi-)stable. When $d = 4$ and $n >(\ge) 5$, hypersurfaces with at most isolated singularities of corank $\le 2$ are (semi-)stable. In particular, the same holds if we substitute "isolated singularities of corank $\le 2$" with "ADE singularities". 
\end{cor}
\begin{rmk}
    Note that according to \cite[Theorem 1.1]{laza2009moduli}, cubic fourfolds with at most ADE singularities are stable. Thus corollary \ref{corollary:cr2} is far from being sharp. However, what is good is that there are only finitely many cases left to be checked. 
\end{rmk}


\subsection{Acknowledgment} I would like to thank Professor Zhiyuan Li for introducing me to this problem.

\section{Preliminaries: Hilbert-Mumford Criterion}
\label{HMC}
Hilbert-Mumford Criterion is an important tool to determine whether a point is semi-stable or stable, for details, one can look them up in the standard reference \cite{mumford1994geometric}, or alternatively \cite{dolgachev2003lectures}, \cite{hoskins2015moduli} and \cite{mukai2003introduction}. For convenience and conciseness, I only state the criterion for projective hypersurfaces. 

A 1-parameter subgroup (1-PS) of $\SL(n+1)$ is a non-trivial algebraic group homomorphism $\lambda:\mathbb{G}_m \to \SL(n+1)$. After a suitable choice of basis, we can always diagonalize this $\lambda$ and obtain
\[
\lambda(t) = \diag(t^{r_0},..,t^{r_n})
\]
where $r_j \in \ZZ$, $r_0\ge r_1\ge..\ge r_n$, $\vec{r} \ne \vec{0}$ and $\Sigma_jr_j = 0$. 

By the Hilbert-Mumford criterion, a hypersurface $H = \{f = 0\}$ is not stable (resp. semi-stable) if and only if there is a linear change of coordinates $\sigma \in \GL(n+1)$ and a 1-PS $\lambda(t) = \diag(t^{r_0},..,t^{r_n})$, such that the limit $\lim_{t\to 0}\lambda(t)^{-1}\sigma f$ exists (resp. exists and is $0$). If we write $\sigma f = \Sigma a_{i_0,..,i_n}x_0^{i_0}..x_n^{i_n}$, then $\lim_{t\to 0}\lambda(t)^{-1}\sigma f$ exists (resp. exists and is $0$) if and only if
\[
\Sigma_jr_ji_j\ge 0,\ \forall a_{i_0,..,i_n}\ne 0\ (resp.\ \Sigma_jr_ji_j> 0,\ \forall a_{i_0,..,i_n}\ne 0). 
\]

For convenience, whenever we take a vector $\vec{r}$, we always assume $r_j \in \ZZ$, $r_0\ge..\ge r_n$, $\vec{r} \ne \vec{0}$ and $\Sigma_jr_j = 0$. Let us define
\[
M_{\ge0}(\vec{r}) = \{f \in \Sigma a_{i_0,..,i_n}x_0^{i_0}..x_n^{i_n}\in \CC[x_0,..,x_n]_d|\Sigma_j r_ji_j \ge 0, \forall a_{i_0,..,i_n}\ne 0\}
\]
and similarly for $M_{>0}(\vec{r})$. 

Therefore, we can state that
\begin{prop}
    A hypersurface $H = V(f)$ is not stable (semi-stable) if and only if there is some $\sigma \in \GL(n+1)$ and an $\vec{r}$ such that $\sigma f \in M_{\ge 0}(\vec{r})$ $(M_{>0}(\vec{r}))$. 
\end{prop}

\section{Proof of theorem \ref{theorem:main1}}
For convenience, I would most of the time ignore "(semi-)" in the statements of the theorems since most of the arguments are the same up to a change of inequality signs. 

The main strategy is to prove that for all $\vec{r}$ and $f \in M_{\ge 0}(\vec{r})$, with particular conditions imposed on $f$, $H = V(f)$ will have some common "bad" singularities. Therefore, for it to be stable, we simply avoid those singularities. 

For any $\vec{r}$, because of the properties we assume it has, we must have $r_0 > 0$ and $r_n < 0$. Thus, the monomial $x_n^d$ will not show up in $f \in M_{\ge 0}(\vec{r})$. As a consequence, we can always write
\[
f = x_n^{d-1}l(x_0,..,x_{n-1})+x_n^{d-2}q(x_0,..,x_{n-1})+..+x_n^{d-j}h_j(x_0,..,x_{n-1})+..+h_d(x_0,..,x_{n-1})
\]
where $l,q,h_j$ are homogeneous polynomials of degree $1$, $2$ and $j$. Therefore, $Q = [0:..:0:1]$ is always a point of $H$. It would be nice to impose conditions on $f$ such that $Q$ becomes a singularity. We first give a criterion on the multiplicity of $Q$. 
\begin{lem}
\label{lemma:criterionofmultiplicity}
    Let $1 \le j \le d-1$ be an integer, suppose $jr_0 + (d-j)r_n < (\le)0$, then $f \in M_{\ge(>0)}(\vec{r})$ will have multiplicity $\ge j+1$ at the point $Q$. In particular, when $r_0 + (d-1)r_n <(\le) 0$, $Q$ is a singular point of $H$. 
\end{lem}
\begin{proof}
    Recall that $r_0\ge..\ge r_n$. Suppose $jr_0 + (d-j)r_n<0$, then for any $k$ integers $0\le i_1 \le .. \le i_k \le n-1$, such that $1 \le k \le j$, we have
    \[
    r_{i_0}+..+r_{i_k} + (d-k)r_n \le jr_0 + (d-j)r_n < 0. 
    \]
    Therefore, since $f \in M_{\ge0}(\vec{r})$, the terms $l,q,h_3,..,h_j$ should all be zero. 
\end{proof}

Now to prove theorem \ref{theorem:main1}, we further impose the condition about the dimension of the singular locus of $H$. 
 
\begin{prop}
\label{proposition:estimate}
    Suppose $f \in M_{\ge(>0)}(\vec{r})$, $s$ is the dimension of the singular locus of $H = V(f)$ and $s \le n-2$. Let $t$ be the largest integer such that $r_t \ge (>) 0$, then we have
    \begin{itemize}
        \item $t \ge \frac{n-s}{2}-1$. 
        \item $r_{n-m-2-s}+(d-1)r_{m+1} \ge (>) 0$, for all $\frac{n-s}{2}-1\le m \le n-2-s$. 
    \end{itemize}
\end{prop}
\begin{proof}
    Since $r_j < 0$ for all $j \ge t+1$, we know that $f$ can always be written as
    \[
    (x_0,..,x_t)^d + (x_0,..,x_t)^{d-1}(x_{t+1},..,x_n)+..+(x_0,..,x_t)^2(x_{t+1},..,x_n)^{d-2} + \Sigma_{j=0}^tx_jf_j(x_{t+1,..,x_n}). 
    \]
    Here, $(x_0,..,x_t)^i(x_{t+1},..,x_n)^{d-i}$ are linear combinations of degree $d$ monomials that are multiplication of degree $i$ monomials in $x_0,..,x_t$ and degree $d-i$ monomials in $(x_{t+1},..,x_n)$. $f_j$ are degree $d-1$ homogeneous polynomials. Using the Jacobian criterion, we find that
    \[
    \{x_0=..=x_t=f_0=..=f_t=0\}
    \]
    is in the singular locus of $H$, which by principal ideal theorem, has dimension $\ge n-2(t+1)$. Therefore, by hypothesis, we have $n-2(t+1) \le s$ and thus
    \[
    t \ge \frac{n-s}{2}-1.
    \]

    For the second statement, we prove by contraction. Suppose there is some $m$ with $\frac{n-s}{2}-1\le m \le n-2-s$ such that $r_{n-m-2-s}+(d-1)r_{m+1} < 0$. Since $r_{n-m-2-s}\ge r_{m+1}$, we must have $r_{m+1} < 0$. Combining the two conditions, we know that $f$ can be written as
    \[
    (x_0,..,x_m)^d + (x_0,..,x_m)^{d-1}(x_{m+1},..,x_n)+..+(x_0,..,x_m)^2(x_{m+1},..,x_n)^{d-2} + \Sigma_{j=0}^{n-m-3-s}x_jf_j(x_{m+1,..,x_n}). 
    \]
    However, the set
    \[
    \{x_0=..=x_m=f_0=..=f_{n-m-3} = 0\}
    \]
    is in the singular locus and has dimension at least $n-(m+1)-(n-m-s-2) = s+1$, which is a contraction. 
\end{proof}
\begin{cor}
\label{corollary:estimate}
    Suppose $f \in M_{\ge(>)0}(\vec{r})$, $s$ is the dimension of its singular locus and $s \le n-2$. Let $t$ denote the same as in proposition \ref{proposition:estimate}, then we have
    \begin{itemize}
        \item $t \ge \frac{n-s}{2}-1$. 
        \item $\frac{1}{d-1}r_0+r_{n-1-s}\ge(>)0$.
        \item $\Sigma_{j=1}^{n-2-s}r_j \ge (>) 0$. (When $s = n-2$, ignore this condition. )
    \end{itemize}
\end{cor}
\begin{proof}
    Since $t \ge \frac{n-s}{2}-1$ from proposition \ref{proposition:estimate}, we have $r_{n-m-2-s} \ge 0$. Therefore
    \[
        r_{n-m-2-s}+r_{m+1} \ge \frac{r_{n-m-2-s}+(d-1)r_{m+1}}{d-1} \ge 0
    \]
    for all $\frac{n-s}{2}-1\le m \le n-2-s$ by proposition \ref{proposition:estimate}. Add up all these inequalities we obtain $\Sigma_{j=1}^{n-2-s}r_j\ge 0$. 
\end{proof}
\begin{proof}[Proof of theorem \ref{theorem:main1}]
    To prove the first statement, it suffices to prove that for all $f \in M_{\ge 0}(\vec{r})$ with singular locus being of dimension $s \le n-2$, the multiplicity of $H = V(f)$ at the point $Q = [0:..:0:1]$ is at least $\frac{d(d-2)}{(s+2)d-(s+3)}$.

    Suppose a positive integer $j^\prime$ is such that $j^\prime r_0+(d-j^\prime)r_n \ge 0$, then $\frac{d-2}{d-1}r_0+\frac{(d-j^\prime)(d-2)}{j^\prime (d-1)}r_n \ge 0$. Add this inequality with $\frac{r_0}{d-1} + r_{n-1-s}$ and $\Sigma_{j=1}^{n-2-s}r_j \ge 0$, we get $\Sigma_{j=0}^{n-1-s}r_j + \frac{(d-j^\prime)(d-2)}{j^\prime(d-1)}r_n \ge 0$. Since $\Sigma_{j=0}^{n-1-s}r_j = -\Sigma_{j=n-s}^nr_j \le -(s+1)r_n$, we obtain
    \[
    \frac{(d-j^\prime)(d-2)}{j^\prime (d-1)}-(s+1)\le 0.
    \]
    Therefore, $j^\prime \ge \frac{d(d-2)}{(s+2)d-(s+3)}$. This means for $j_0 = \frac{d(d-2)}{(s+2)d-(s+3)} - 1$, we must have $j_0r_0+(d-j_0)r_n < 0$ and by lemma \ref{lemma:criterionofmultiplicity}, the multiplicity of $H$ at $Q$ is at least $\frac{d(d-2)}{(s+2)d-(s+3)}$. 

    For the second statement, Suppose a positive integer $j^\prime$ is such that $j^\prime r_0+(d-j^\prime)r_n \ge 0$. Since $0\le j^\prime r_0 + (d-j^\prime)r_n = -j^\prime \Sigma_{j=1}^nr_j + (d-j^\prime)r_n \le (-j^\prime n+d-j^\prime)r_n$, we obtain $j^\prime \ge \frac{d}{n+1}$. Applying lemma \ref{lemma:criterionofmultiplicity} again and we are done. 
\end{proof}

\section{Proof of Theorem \ref{theorem:main2}}
In the following, as the conditions in the statement of theorem \ref{theorem:main2}, we should always assume $s = 0$ and $\delta = 2$. 

The following lemma is the key ingredient
\begin{lem}
\label{lemma:rank}
    Suppose $f \in M_{\ge 0(>0)}(\vec{r})$, and write $f$ as
    \[
    f = x_n^{d-1}l(x_0,..,x_{n-1})+x_n^{d-2}q(x_0,..,x_{n-1})+..+x_n^{d-j}h_j(x_0,..,x_{n-1})+..+h_d(x_0,..,x_{n-1}).
    \]
    Let $m_0 = \min\{m \in \ZZ|m >(\ge)\frac{2(n+1)}{d}-1\}$, then we always have $\rank(q) \le m_0$. 
\end{lem}
\begin{proof}
    Write $q = \Sigma_{i,j=0}^{n-1}a_{ij}x_ix_j$ with $a_{ij} = a_{ji}$. Then $\rank(q) = \rank(a_{ij})$. I claim that there must be some integer $0 \le u \le m_0$, such that $r_u + r_{m_0-u} < -(d-2)r_n$. Otherwise, Suppose $r_u + r_{m_0-u}\ge -(d-2)r_n$ for all $0 \le u \le m_0$. Add up all these inequalities and divide it by $2$, we obtain $\Sigma_{j=0}^{m_0}r_j \ge \frac{(m_0+1)(d-2)}{2}(-r_n) > (n-m_0)(-r_n)$. Meanwhile, since $\Sigma_{j=0}^{m_0}r_j = -\Sigma_{j=m_0+1}^nr_j \le (n-m_0)(-r_n)$, we get a contraction. 

    Now since $r_u + r_{m_0-u} < -(d-2)r_n$, we obtain $r_i + r_j + (d-2)r_n < 0$ for all $i \ge u$ and $j \ge m_0 -u$. Therefore, the matrix $A = (a_{ij})$ is of the form
    \[
    \begin{pmatrix}
        B_{u\times(m_0-u)} & C_{u\times (n-m_0+u)}\\
        D_{(n-u)\times (m_0 - u)} & 0. 
    \end{pmatrix}
    \]

    Therefore, $\rank(q) \le \rank(C) + \rank(D) \le u+(m_0 - u) = m_0$. 
\end{proof}

Therefore, if we can prove that $Q$ is always the singularity of $H$, then theorem \ref{theorem:main2} will be proven once we apply lemma \ref{lemma:rank}. However, this is not always the case and we need a trick for special cases. 

\begin{prop}
\label{proposition:final}
    Suppose $f \in M_{\ge 0(>0)}(\vec{r})$ and $s = 0$. When $d \ge 4$, $Q$ is a singularity of $H$. When $d = 3$ and $f \in M_{>0}(\vec{r})$, $Q$ is a singularity of $H$. 
\end{prop}
\begin{proof}
    Note that in the proof of theorem \ref{theorem:main1}, we obtain that the multiplicity of $H$ at the point $Q$ is $\ge(>) \frac{d(d-2)}{(s+2)d-(s+3)}$. Here, we assume $s = 0$ and thus the multiplicity at $Q$ is $\ge(>)\frac{d(d-2)}{2d-3}$. Also note that under the hypothesis of this proposition, we always have the multiplicity of $Q$ to be $\ge 2$. 
\end{proof}

What is left is the case where $f \in M_{\ge 0}(\vec{r})$ and $d = 3$. In fact, we can prove the following
\begin{prop}
\label{proposition:finaltrick}
    Suppose $f \in M_{\ge 0}(\vec{r})$, $d = 3$ and $s = 0$. We can always find a linear change of coordinates $\sigma \in \GL(n+1)$ and a 1-PS $\vec{r^\prime}$, satisfying the same conditions as $\vec{r}$, such that $r_0^\prime + 2r_n^\prime < 0$ and $\sigma f \in M_{\ge 0}(\vec{r^\prime})$. In particular, $[0:..:0:1]$ is a singularity of $V(\sigma f)$ by lemma \ref{lemma:criterionofmultiplicity}. 
\end{prop}

We first need a little lemma
\begin{lem}
\label{lemma:final}
    Suppose $\vec{r}$ is such that $r_0 + 2r_n \ge 0$, then we must have $r_{n-1}<0, r_{n-2}\le 0$. When $r_{n-2} = 0$, we must have $r_1 = .. = r_{n-2} = 0$. 
\end{lem}
\begin{proof}
    Suppose $r_{n-1} \ge 0$, we must have
    \[
    0 = \Sigma_{j=0}^nr_j \ge -2r_n +\Sigma_{j=1}^nr_j = \Sigma_{j=1}^{n-2}r_j + (r_{n-1}-r_n) > 0,
    \]
    which is a contraction. The same argument applies to $r_{n-2} > 0$ and $r_{n-2} = 0$.  
\end{proof}

\begin{proof}[Proof of proposition \ref{proposition:finaltrick}]
    When $r_0 + 2r_n < 0$, there is no need to prove. Thus, just assume $r_0 + 2r_n \ge 0$. 

    By lemma \ref{lemma:final}, we have $r_{n-2}<0$ or $r_{n-2} = 0$. 

    When $r_{n-2}<0$, recall that in the proof of corollary \ref{corollary:estimate}, we have $r_{n-m-2} + r_{m+1} \ge 0$ for all $\frac{n}{2}-1\le m\le n-3$. But here, we will further have $r_{1} + r_{n-2} > r_1 + 2r_{n-2} \ge 0$ where the latter inequality follows from proposition \ref{proposition:estimate}. Therefore, adding up all the inequalities $r_{n-m-2} + r_{m+1} \ge 0$ with the exception $r_1 + r_{n-2} > 0$ will give us $\Sigma_{j=1}^{n-2}r_j>0$. Since $r_0+r_{n-1}+r_n \ge r_0 + 2r_n \ge 0$ by hypothesis, we will get a contraction $0 = \Sigma_{j=0}^nr_j >0$. Therefore, $r_{n-2}<0$ will not happen. 

    When $r_{n-2} = 0$. By lemma \ref{lemma:final}, we would have $r_1 = .. = r_{n-2} = 0$. Therefore, $f$ can be written as
    \[
    f = c(x_0,..,x_{n-2}) + x_0q(x_{n-1},x_n) + x_nx_0l_1(x_0,..,x_{n-2}) + x_{n-1}x_0l_2(x_0,..,x_{n-2})
    \]
    where $c$, $q$ and $l_i$ are homogeneous polynomials of degree $3$, $2$ and $1$. After a linear change of coordinates in $x_{n-1}$ and $x_n$, we may assume $q = x_{n-1}x_n$ or $q = x_{n-1}^2$ according to $\rank(q)$. No matter which case it is, we can write
    \[
    f = x_nx_0l(x_0,..,x_{n-1}) + c(x_0,..,x_{n-1}). 
    \]

    Substitute $l(x_0,..,x_{n-1})$ with $x_1$, we get
    \[
    \sigma f = x_nx_0x_1 + c(x_0,..,x_{n-1}).
    \]

    Now let $\vec{r^\prime} = (1,1,0,..,0,-2)$, we have $\sigma f \in M_{\ge 0}(\vec{r^\prime})$. 
\end{proof}

Therefore, theorem \ref{theorem:main2} will be a direct consequence of lemma \ref{lemma:final}, proposition \ref{proposition:final} and proposition \ref{proposition:finaltrick}. 

\bibliography{reference}
\end{document}